\documentclass{amsart}
\usepackage{amsmath}
\usepackage{amssymb}
\usepackage{amsthm}
\usepackage{epsfig}
\usepackage{amsfonts}
\usepackage{amscd}
\usepackage{mathrsfs}
\usepackage{enumerate}
\usepackage{latexsym}
\usepackage{url}

\newtheorem{theorem}{Theorem}[section]
\newtheorem{lemma}[theorem]{Lemma}

\newtheorem{theoremandnotation}[theorem]{Theorem and Notation}

\newenvironment{proof of claim}{\noindent\textbf{Proof of the claim.}}{\hfill{$\square$}\newline}

\theoremstyle{definition}
\newtheorem{definition}[theorem]{Definition}

\theoremstyle{remark}

\numberwithin{equation}{section}

\begin{document}

\title[On closed non-vanishing ideals in $C_B(X)$ II; compactness properties]{On closed non-vanishing ideals in $C_B(X)$ II; compactness properties}

\author[A. Khademi and M.R. Koushesh]{A. Khademi and M.R. Koushesh$^*$}

\address{\textbf{[First author]} Department of Mathematical Sciences, Isfahan University of Technology, Isfahan 84156--83111, Iran.}

\email{a.khademi@math.iut.ac.ir}

\address{\textbf{[Second author]} Department of Mathematical Sciences, Isfahan University of Technology, Isfahan 84156--83111, Iran, and, School of Mathematics, Institute for Research in Fundamental Sciences (IPM), P.O. Box: 19395--5746, Tehran, Iran.}

\email{koushesh@cc.iut.ac.ir}

\thanks{$^*$Corresponding author}

\thanks{The research of the second author was in part supported by a grant from IPM (No. 96030416).}



\begin{abstract}
For a completely regular space $X$, let $C_B(X)$ be the normed algebra of all bounded continuous scalar-valued mappings on $X$ equipped with pointwise addition and multiplication and the supremum norm and let $C_0(X)$ be its subalgebra consisting of mappings vanishing at infinity. For a non-vanishing closed ideal $H$ of $C_B(X)$ we study properties of its spectrum $\mathfrak{sp}(H)$ which may be characterized as the unique locally compact (Hausdorff) space $Y$ such that $H$ and $C_0(Y)$ are isometrically isomorphic. We concentrate on compactness properties of $\mathfrak{sp}(H)$ and find necessary and sufficient (algebraic) conditions on $H$ such that the spectrum $\mathfrak{sp}(H)$ satisfies (topological) properties such as the Lindel\"{o}f property, $\sigma$-compactness, countable compactness, pseudocompactness and paracompactness.
\end{abstract}

\maketitle

%\tableofcontents

\section{Introduction}

A \textit{space} means a topological space. We adopt the definitions of \cite{E}, in particular, completely regular spaces, compact spaces, locally compact spaces, $\sigma$-compact spaces, countably compact spaces and paracompact spaces are Hausdorff, Lindel\"{o}f spaces are regular, and pseudocompact spaces are completely regular. The field of scalars, which is fixed throughout, is $\mathbb{R}$ or $\mathbb{C}$, and is denoted by $\mathbb{F}$. For a space $X$ the set of all bounded continuous scalar-valued mappings on $X$ is denoted by $C_B(X)$. The set $C_B(X)$ is a Banach algebra when equipped with pointwise addition and multiplication and the supremum norm. The normed subalgebra of $C_B(X)$ consisting of mappings which vanish at infinity is denoted by $C_0(X)$.

In \cite{K4}, motivated by previous studies in \cite{K1}--\cite{K3}, the second author has studied closed ideals $H$ of $C_B(X)$ which are non-vanishing, in the sense that, for every element $x$ of $X$ there is an element of $H$ which does not vanish at $x$. (Here $X$ is assumed to be a completely regular space which is not necessarily a locally compact space.) This has been done by studying the spectrum $\mathfrak{sp}(H)$ of $H$, i.e., the unique locally compact space $Y$ such that $H$ and $C_0(Y)$ are isometrically isomorphic. The spectrum $\mathfrak{sp}(H)$ has been constructed as a subspace of the Stone--\v{C}ech compactification of $X$. The study in \cite{K4} has been continued in \cite{KK} where we have considered connected properties of $\mathfrak{sp}(H)$. This includes finding necessary and sufficient (algebraic) conditions on $H$ such that the spectrum $\mathfrak{sp}(H)$ satisfies connectedness (topological) properties such as locally connectedness, total disconnectedness, zero-dimensionality, strong zero-dimensionality, total separatedness and extremal disconnectedness. The purpose of this article is to undertake a similar study by finding necessary and sufficient conditions on $H$ such that the spectrum $\mathfrak{sp}(H)$ satisfies compactness properties such as the Lindel\"{o}f property, $\sigma$-compactness, countable compactness, pseudocompactness and paracompactness.

We use the theory of the Stone--\v{C}ech compactification as the tool. We mention the basic definitions in the following and refer the reader to \cite{GJ} for further information.

\subsection*{The Stone--\v{C}ech compactification}

Let $X$ be a completely regular space. A \textit{compactification} of $X$ is a compact space $\alpha X$ which contains $X$ as a dense subspace. The \textit{Stone--\v{C}ech compactification} of $X$ (denoted by $\beta X$) is the ``largest'' compactification of $X$ which is characterized by the property that every bounded continuous mapping $f:X\rightarrow\mathbb{F}$ is extendible to a continuous mapping $F:\beta X\rightarrow\mathbb{F}$. For every bounded continuous mapping $f:X\rightarrow\mathbb{F}$ we denote by $f_\beta$ or $f^\beta$ the (unique) continuous extension of $f$ to $\beta X$. Note that for every continuous mappings $f,g:X\rightarrow\mathbb{F}$ and a scalar $r$ we have $(f+g)_\beta=f_\beta+g_\beta$, $(fg)_\beta=f_\beta g_\beta$, $(rf)_\beta=rf_\beta$, $\overline{f}_\beta=\overline{f_\beta}$, $|f|_\beta=|f_\beta|$ and $\|f_\beta\|=\|f\|$.

\section{Preliminaries}

In \cite{K4}, the second author has studied closed non-vanishing ideals $H$ of $C_B(X)$, where $X$ is a completely regular space, by studying their spectrum $\mathfrak{sp}(H)$. (By $H$ being \textit{non-vanishing} it is meant that for every element $x$ of $X$ there is an element of $H$ which does not vanish at $x$.) The spectrum $\mathfrak{sp}(H)$ has a simple description as a subspace of the Stone--\v{C}ech compactification of $X$, as described in the following theorem quoted from \cite{K4}.

For a mapping $f:Y\rightarrow\mathbb{F}$ the \textit{cozeroset} of $f$ is defined as the set of all $y$ in $Y$ such that $f(y)\neq 0$, and is denoted by $\mathrm{coz}(f)$.

\begin{theoremandnotation}\label{TRES}
Let $X$ be a completely regular space. Let $H$ be a non-vanishing closed ideal in $C_B(X)$. The spectrum of $H$ is defined as
\[\mathfrak{sp}(H)=\bigcup_{h\in H}\mathrm{coz}(h_\beta).\]
The spectrum $\mathfrak{sp}(H)$ is an open subspace of $\beta X$, and thus, is a locally compact space. Also, $\mathfrak{sp}(H)$ contains $X$ as a dense subspace. Further, $H$ and $C_0(\mathfrak{sp}(H))$ are isometrically isomorphic. The spectrum $\mathfrak{sp}(H)$ is therefore characterized as the unique locally compact space $Y$ such that $H$ and $C_0(Y)$ are isometrically isomorphic. In particular, $\mathfrak{sp}(H)$ coincides with the Gelfand spectrum of $H$, in the usual sense, when the field of scalars is $\mathbb{C}$.
\end{theoremandnotation}

The simple representation of $\mathfrak{sp}(H)$ in the above theorem enables studying its properties by relating (topological) properties of $\mathfrak{sp}(H)$ to (algebraic) properties of $H$. This has been done in \cite{K4} and later in \cite{KK}, where in the latter the authors have concentrated on the study of various connectedness properties of $\mathfrak{sp}(H)$. Among results of this type we mention the following two theorems quoted from \cite{K4} and \cite{KK}, respectively.

\begin{theorem}\label{HJGS}
Let $X$ be a completely regular space. Let $H$ be a non-vanishing closed ideal in $C_B(X)$. The following are equivalent:
\begin{itemize}
\item[\rm(1)] $\mathfrak{sp}(H)$ is connected.
\item[\rm(2)] $H$ is indecomposable, that is
\[H\neq I\oplus J\]
for any non-zero ideals $I$ and $J$ of $C_B(X)$.
\end{itemize}
\end{theorem}

\begin{theorem}\label{KHD}
Let $X$ be a completely regular space. Let $H$ be a non-vanishing closed ideal in $C_B(X)$. The following are equivalent:
\begin{itemize}
\item[\rm(1)] $\mathfrak{sp}(H)$ is locally connected.
\item[\rm(2)] For every closed subideal $G$ of $H$ we have
\[G=\overline{\bigoplus_{i\in I}G_i},\]
where $\{G_i:i\in I\}$ is a collection of indecomposable closed ideals of $C_B(X)$.
\end{itemize}
Here the bar denotes the closure in $C_B(X)$.
\end{theorem}

The purpose of this article is to study compactness properties of the spectrum $\mathfrak{sp}(H)$ of non-vanishing closed ideals $H$ of $C_B(X)$. Examples of non-vanishing closed ideals of $C_B(X)$ are described in \cite{K4}. (See \cite{FK1}--\cite{FK2} and \cite{K1}--\cite{K3} for other relevant results and examples.)

\section{Compactness properties of the spectrum}

The purpose of this section is to study non-vanishing closed ideals $H$ of $C_B(X)$, where $X$ is a completely regular space, by relating algebraic properties of $H$ to topological properties of the spectrum $\mathfrak{sp}(H)$. We will concentrate on compactness properties, such as the Lindel\"{o}f property, $\sigma$-compactness, countable compactness, pseudocompactness and paracompactness.

We will need a definition and a few lemmas from \cite{K4} and \cite{KK}. These we list in the following and refer the reader to \cite{K4} and \cite{KK} for missing proofs.

\begin{definition}\label{GHG}
Let $X$ be a completely regular space. For an ideal $G$ of $C_B(X)$ let
\[\lambda_GX=\bigcup_{g\in G}\mathrm{coz}(g_\beta).\]
\end{definition}

Note that by Theorem \ref{TRES}, for a non-vanishing closed ideal $H$ of $C_B(X)$ (where $X$ is a completely regular space) we have $\mathfrak{sp}(H)=\lambda_HX$.

The next three lemmas (with Lemma \ref{JH} slightly modified) may be found in \cite{KK}.

\begin{lemma}\label{UEED}
Let $X$ be a completely regular space. Let $G$ be an ideal of $C_B(X)$. Then
\[\lambda_{\overline{G}}X=\lambda_GX.\]
Here the bar denotes the closure in $C_B(X)$.
\end{lemma}

\begin{lemma}\label{HOOF}
Let $X$ be a completely regular space. Let $G_1$ and $G_2$ be closed ideals of $C_B(X)$ such that $\lambda_{G_1}X=\lambda_{G_2}X$. Then $G_1=G_2$.
\end{lemma}

\begin{lemma}\label{JH}
Let $X$ be a completely regular space. Let $\{G_i:i\in I\}$ be a collection of ideals of $C_B(X)$. Then
\[\lambda_{\sum_{i\in I}G_i}X=\bigcup_{i\in I}\lambda_{G_i}X.\]
\end{lemma}

The following lemma is proved in \cite{K4}.

\begin{lemma}\label{JJHF}
Let $X$ be a completely regular space. Let $f$ be in $C_B(X)$. Let $f_1,f_2,\ldots$ be a sequence in $C_B(X)$ such that
\[|f|^{-1}\big([1/n,\infty)\big)\subseteq |f_n|^{-1}\big([1,\infty)\big)\]
for every positive integer $n$. Then $g_nf_n\rightarrow f$ for some sequence $g_1,g_2,\ldots$ in $C_B(X)$.
\end{lemma}

Our first theorem deals with the Lindel\"{o}f property and $\sigma$-compactness. Recall that a regular space is called \textit{Lindel\"{o}f} if every open cover of it has a countable subcover. Also, a Hausdorff space is called \textit{$\sigma$-compact} if it is a countable union of compact subspaces. It is clear that every $\sigma$-compact regular space is a Lindel\"{o}f space. The converse also holds if the space is locally compact. (See Exercise 3.8.C(b) of \cite{E}.)

In the sequel we denote the ideal generated by an element $g$ of $C_B(X)$ by $\langle g\rangle$ (where $X$ is a completely regular space).

\begin{lemma}\label{JHFF}
Let $X$ be a completely regular space. For a closed ideal $G$ of $C_B(X)$ the following are equivalent:
\begin{itemize}
\item[(1)] $\lambda_GX$ is Lindel\"{o}f.
\item[(2)] $\lambda_GX$ is $\sigma$-compact.
\item[(3)] $G=\overline{\langle g\rangle}$ for some $g$ in $G$.
\end{itemize}
Here the bar denotes the closure in $C_B(X)$.
\end{lemma}

\begin{proof}
Note that (1) and (2) are equivalent, as $\lambda_GX$, by its definition, is open in (the compact space) $\beta X$, and is therefore locally compact.

(1) \textit{implies} (3). By definition of $\lambda_GX$ the collection $\{\mathrm{coz}(g^\beta):g\in G\}$ forms an open cover for $\lambda_GX$. Therefore
\begin{equation}\label{KHD}
\lambda_GX=\bigcup_{i=1}^\infty\mathrm{coz}(g_i^\beta)
\end{equation}
for elements $g_1,g_2,\ldots$ in $G$. Note that $\mathrm{coz}(g_i^\beta)=\lambda_{\langle g_i\rangle}X$ for every positive integer $i$. Thus, using Lemmas \ref{UEED} and \ref{JH}, from (\ref{KHD}) we have
\[\lambda_GX=\bigcup_{i=1}^\infty\lambda_{\langle g_i\rangle}X=\lambda_{\sum_{i=1}^\infty\langle g_i\rangle}X=\lambda_{\overline{\sum_{i=1}^\infty\langle g_i\rangle}}X.\]
Lemma \ref{HOOF} now implies that
\begin{equation}\label{JGD}
G=\overline{\sum_{i=1}^\infty\langle g_i\rangle}.
\end{equation}
Without any loss of generality we may assume that $g_i$ is non-zero for every positive integer $i$. Let
\[g=\sum_{i=1}^\infty\frac{h_i}{2^i\|h_i\|},\]
where $h_i=|g_i|^2=g_i\overline{g_i}$ (which is an element of $G$) for every positive integer $i$. Then $g$ is a well defined mapping which is continuous by the Weierstrass $M$-test. Thus, in particular, $g$ is in $G$, as is the limit of a sequence in $G$ and $G$ is closed in $C_B(X)$. We show that $G=\overline{\langle g\rangle}$. It is clear that $\overline{\langle g\rangle}\subseteq G$. To show the reverse inclusion, by (\ref{JGD}), it suffices to show that $g_i$ is in $\overline{\langle g\rangle}$ for all positive integers $i$. For this purpose, fix some positive integer $i$. By the definition of $g$ we have
\[|g_i|^{-1}\big([1/n,\infty)\big)\subseteq |u_n|^{-1}\big([1,\infty)\big)\]
for all positive integers $n$ where $u_n=n^22^i\|h_i\|g$ (which is an element of $\langle g\rangle$). By Lemma \ref{JJHF} then $u_nf_n\rightarrow g_i$ for some sequence $f_1,f_2,\ldots$ in $C_B(X)$. Therefore, $g_i$ is the limit of a sequence in $\langle g\rangle$ and is therefore in $\overline{\langle g\rangle}$. This shows that $G\subseteq\overline{\langle g\rangle}$.

(3) \textit{implies} (2). Note that
\[\lambda_GX=\lambda_{\overline{\langle g\rangle}}X=\lambda_{\langle g\rangle}X\]
by Lemma \ref{UEED}, and $\lambda_{\langle g\rangle}X=\mathrm{coz}(g^\beta)$. But the latter is $\sigma$-compact, as
\[\mathrm{coz}(g^\beta)=\bigcup_{n=1}^\infty |g^\beta|^{-1}\big([1/n,\infty)\big).\]
\end{proof}

The following lemma is quoted from \cite{KK}.

\begin{lemma}\label{HJGF}
Let $X$ be a completely regular space. Let $H$ be a non-vanishing closed ideal in $C_B(X)$. Then the open subspaces of $\mathfrak{sp}(H)$ are exactly those of the form $\lambda_GX$ where $G$ is a closed subideal of $H$. Specifically, for an open subspace $U$ of $\mathfrak{sp}(H)$ we have
\[U=\lambda_GX,\]
where
\[G=\{g\in H:g_\beta|_{\beta X\setminus U}=0\}.\]
\end{lemma}

In \cite{K4} it is proved that the spectrum $\mathfrak{sp}(H)$ of a non-vanishing closed ideal $H$ of $C_B(X)$ is $\sigma$-compact if and only if $H$ is $\sigma$-generated. (As usual, $X$ is a completely regular space.) The next theorem, in part, is an extension of this result.

\begin{theorem}\label{HF}
Let $X$ be a completely regular space. Let $H$ be a non-vanishing closed ideal in $C_B(X)$. The following are equivalent:
\begin{itemize}
\item[(1)] $\mathfrak{sp}(H)$ is Lindel\"{o}f.
\item[(2)] $\mathfrak{sp}(H)$ is $\sigma$-compact.
\item[(3)] $H=\overline{\langle h\rangle}$ for some $h$ in $H$.
\item[(4)] For any collection $\{H_i:i\in I\}$ of closed ideals of $C_B(X)$, if $H=\overline{\sum_{i\in I} H_i}$ then $H=\overline{\sum_{j=1}^\infty H_{i_j}}$ for some $i_1,i_2,\ldots$ in $I$.
\end{itemize}
Here the bar denotes the closure in $C_B(X)$.
\end{theorem}

\begin{proof}
The equivalence of (1), (2) and (3) follows from Lemma \ref{JHFF}.

(1) \textit{implies} (4). Let $\{H_i:i\in I\}$ be a collection of ideals of $C_B(X)$ such that \[H=\overline{\sum_{i\in I} H_i}.\]
Then, by Lemmas \ref{UEED} and \ref{JH} we have
\[\lambda_HX=\lambda_{\overline{\sum_{i\in I} H_i}}X=\lambda_{\sum_{i\in I} H_i}X=\bigcup_{i\in I}\lambda_{H_i}X.\]
Therefore, since $\lambda_HX$ ($=\mathfrak{sp}(H)$) is Lindel\"{o}f, we have
\begin{equation}\label{JHBG}
\lambda_HX=\bigcup_{j=1}^\infty\lambda_{H_{i_j}}X
\end{equation}
for some $i_1,i_2,\ldots$ in $I$. But, again by Lemmas \ref{UEED} and \ref{JH} we have
\[\bigcup_{j=1}^\infty\lambda_{H_{i_j}}X=\lambda_{\sum_{j=1}^\infty H_{i_j}}X=\lambda_{\overline{\sum_{j=1}^\infty H_{i_j}}}X,\]
which, together with (\ref{JHBG}), by Lemma \ref{HOOF} we have
\[H=\overline{\sum_{j=1}^\infty H_{i_j}}.\]

(4) \textit{implies} (1). Note that by Lemma \ref{HJGF} the open subspaces of $\mathfrak{sp}(H)$ are of the form $\lambda_GX$ where $G$ is a closed subideal of $H$. An argument similar to the one we used to prove that (1) implies (4) now completes the proof.
\end{proof}

Our next theorem deals with countable compactness. Recall that a Hausdorff space $X$ is \textit{countably compact} if every countable open cover of $X$ has a finite subcover, or, equivalently, if for every sequence $U_1\subseteq U_2\subseteq\cdots$ of proper open subspaces of $X$ we have $X\neq\bigcup_{n=1}^\infty U_n$. (See Theorem 3.10.2 of \cite{E}.)

\begin{theorem}\label{JGDSS}
Let $X$ be a completely regular space. Let $H$ be a non-vanishing closed ideal in $C_B(X)$. The following are equivalent:
\begin{itemize}
\item[(1)] $\mathfrak{sp}(H)$ is countably compact.
\item[(2)] For every sequence $H_1,H_2,\ldots$ of closed subideals of $H$, if $H=\overline{\sum_{n=1}^\infty H_n}$ then $H=\overline{\sum_{n=1}^m H_n}$ for some positive integer $m$.
\item[(3)] For every sequence $H_1\subseteq H_2\subseteq\cdots$ of proper closed subideals of $H$ we have \[H\neq\overline{\bigcup_{n=1}^\infty H_n}.\]
\end{itemize}
Here the bar denotes the closure in $C_B(X)$.
\end{theorem}

\begin{proof}
The proof uses the same methods as in the proof of Theorem \ref{HF}. Observe that for closed subideals $H_1$ and $H_2$ of $H$ if $\lambda_{H_1}X\subseteq\lambda_{H_2}X$ then $H_1\subseteq H_2$, as
\[\lambda_{H_2}X=\lambda_{H_1}X\cup\lambda_{H_2}X=\lambda_{H_1+H_2}X=\lambda_{\overline{H_1+H_2}}X\]
by Lemmas \ref{UEED} and \ref{JH}, which implies that $H_2=\overline{H_1+H_2}$ by Lemma \ref{HOOF}.
\end{proof}

Our next theorem deals with pseudocompactness. Recall that a completely regular space is \textit{pseudocompact} if every continuous scalar valued mapping on $X$ is bounded, or, equivalently, if there is no infinite discrete collection of open subspaces in $X$. (A collection $\mathscr{A}$ of subsets of a space $X$ is called \textit{discrete} if every element of $X$ has a neighborhood which intersects at most one element from $\mathscr{A}$.) Completely regular countably compact spaces are pseudocompact and normal pseudocompact spaces are countably compact. (See Theorems 3.10.20 and 3.10.21 of \cite{E}.)

The following two lemmas are quoted from \cite{KK}.

\begin{lemma}\label{HGFD}
Let $X$ be a completely regular space. Let $H$ be a non-vanishing closed ideal in $C_B(X)$. For a closed subideal $M$ of $H$ the following are equivalent:
\begin{itemize}
\item[(1)] $M$ is a maximal closed subideal of $H$.
\item[(2)] There is some $x$ in $\lambda_HX$ such that
\[\lambda_MX=\lambda_HX\setminus\{x\}.\]
\end{itemize}
\end{lemma}

\begin{lemma}\label{HFDS}
Let $X$ be a completely regular space. Let $G_1$ and $G_2$ be ideals of $C_B(X)$. Then
\[\lambda_{G_1\cap G_2}X=\lambda_{G_1}X\cap\lambda_{G_2}X.\]
\end{lemma}

\begin{theorem}\label{JHJ}
Let $X$ be a completely regular space. Let $H$ be a non-vanishing closed ideal in $C_B(X)$. The following are equivalent:
\begin{itemize}
\item[(1)] $\mathfrak{sp}(H)$ is pseudocompact.
\item[(2)] $H$ has no subideal of the form
\[\bigoplus_{i\in I} H_i\]
where $\{H_i:i\in I\}$ is an infinite (faithfully indexed) collection of closed subideals of $H$ such that for every maximal closed subideal $M$ of $H$ there is an $f$ in $H\setminus M$ such that $\langle f\rangle\cap H_i=0$ except for at most one $i$ in $I$.
\end{itemize}
\end{theorem}

\begin{proof}
(1) \textit{implies} (2). Suppose that $H$ contains a subideal of the form $\bigoplus_{i\in I} H_i$ with properties as indicated in (2). For all $i$ in $I$ let $U_i=\lambda_{H_i}X$. We check that $\mathscr{U}=\{U_i:i\in I\}$ is an infinite discrete collection of open subspaces of $\lambda_HX$ ($=\mathfrak{sp}(H)$). This will prove that $\lambda_HX$ is not pseudocompact. It is clear that $\mathscr{U}$ is infinite (as the number of $H_i$'s are infinite, and there is a one to one correspondence between $H_i$'s and $\lambda_{H_i}X$'s by Lemma \ref{HOOF}) and $\mathscr{U}$ consists of open subspaces of $\lambda_HX$. Let $x$ be in $\lambda_HX$. Let $M$ be a closed subideal of $H$ such that $\lambda_MX=\lambda_HX\setminus\{x\}$, which exists by Lemma \ref{HJGF}. By Lemma \ref{HGFD} then $M$ is a maximal closed subideal of $H$. Using our assumption, there is an $f$ in $H\setminus M$ such that $\langle f\rangle\cap H_i=0$ except for at most one $i$ in $I$. Note that $f_\beta(x)\neq0$, as otherwise, using Lemma \ref{UEED}, we have
\[\lambda_{\overline{\langle f\rangle}}X=\lambda_{\langle f\rangle}X=\mathrm{coz}(f_\beta)\subseteq\lambda_MX,\]
which (arguing as in the proof of Theorem \ref{JGDSS}) implies that $\overline{\langle f\rangle}\subseteq M$ (and in particular, $f$ is in $M$), which contradicts the choice of $f$. Let $U=\mathrm{coz}(f_\beta)$. Then $U$ is an open neighborhood of $x$ in $\lambda_HX$ which intersects $U_i$'s for at most one $i$ in $I$, as
\[U\cap U_i=\lambda_{\langle f\rangle}X\cap\lambda_{H_i}X=\lambda_{\langle f\rangle\cap H_i}X\]
by Lemma \ref{HFDS}, and $U\cap U_i\neq\emptyset$ implies that $\langle f\rangle\cap H_i\neq 0$ for all $i$ in $I$.

(2) \textit{implies} (1). Suppose that $\mathfrak{sp}(H)$ ($=\lambda_HX$) is not pseudocompact. Let $\mathscr{U}=\{U_i:i\in I\}$ be an infinite (faithfully indexed) discrete collection of open subspaces in $\lambda_HX$. For every $i$ in $I$ let $H_i$ be a closed subideal of $H$ such that $U_i=\lambda_{H_i}X$, which exists by Lemma \ref{HJGF}. Consider the collection $\{H_i:i\in I\}$ which is faithfully indexed, as $\mathscr{U}$ is (and there is a one to one correspondence between $H_i$'s and $U_i$'s by Lemma \ref{HOOF}). For a fixed $i_0$ in $I$ note that
\begin{eqnarray*}
\lambda_{H_{i_0}\cap\sum_{i_0\neq i\in I}H_i}X&=&\lambda_{H_{i_0}}X\cap\lambda_{\sum_{i_0\neq i\in I}H_i}X\\&=&\lambda_{H_{i_0}}X\cap\bigcup_{i_0\neq i\in I}\lambda_{H_i}X=U_{i_0}\cap\bigcup_{i_0\neq i\in I}U_i=\emptyset
\end{eqnarray*}
by Lemmas \ref{JH} and \ref{HFDS}. This implies that $H_{i_0}\cap\sum_{i_0\neq i\in I}H_i=0$. Therefore
\[\sum_{i\in I}H_i=\bigoplus_{i\in I}H_i,\]
which is a subideal of $H$. Now, let $M$ be a maximal closed subideal of $H$. By Lemma \ref{HGFD} then $\lambda_MX=\lambda_HX\setminus\{x\}$ for some $x$ in $\lambda_HX$. Since $\mathscr{U}$ is discrete, there is an open neighborhood $U$ of $x$ in $\lambda_HX$ which intersects at most one element of $\mathscr{U}$. Let $F:\beta X\rightarrow [0,1]$ be a continuous mapping such that $F(x)=1$ and $F|_{\beta X\setminus U}=0$. (Such an $F$ exists, as $\beta X$ is completely regular.) Let $h$ be an element of $H$ such that $h_\beta(x)\neq 0$. (Such an $h$ exists by the definition of $\lambda_HX$.) Let $f=hF|_X$. Then $f$ is in $H$, but is not in $M$ (as otherwise, $x$ would be in $\lambda_MX$, since $f_\beta(x)=h_\beta(x)F(x)\neq 0$). By Lemma \ref{HFDS} we have
\[\lambda_{\langle f\rangle\cap H_i}X=\lambda_{\langle f\rangle}X\cap\lambda_{H_i}X=\mathrm{coz}(f_\beta)\cap U_i\]
and $\mathrm{coz}(f_\beta)\subseteq\mathrm{coz}(F)\subseteq U$, thus $\lambda_{\langle f\rangle\cap H_i}X$ is contained in $U\cap U_i$ for all $i$ in $I$. Therefore $\lambda_{\langle f\rangle\cap H_i}X$ is empty except for at most one $i$ in $I$. Thus $\langle f\rangle\cap H_i$ is trivial except for at most one $i$ in $I$.
\end{proof}

Our concluding theorem deals with paracompactness. Recall that a Hausdorff space is \textit{paracompact} if every open cover of it has a locally finite open refinement. It is known that every locally compact paracompact space may be represented as the union of a collection of pairwise disjoint open (and thus also closed) Lindel\"{o}f subspaces. (See Theorem 5.1.27 of \cite{E}.) The converse statement also holds, i.e., a space which is the union of a collection of pairwise disjoint open Lindel\"{o}f subspaces is paracompact. (See Theorem 5.1.30 of \cite{E}. Note that Lindel\"{o}f spaces are paracompact by Theorem 5.1.2 of \cite{E}.)

\begin{theorem}\label{JGDS}
Let $X$ be a completely regular space. Let $H$ be a non-vanishing closed ideal in $C_B(X)$. The following are equivalent:
\begin{itemize}
\item[(1)] $\mathfrak{sp}(H)$ is paracompact.
\item[(2)] $H$ is representable as
\[H=\overline{\bigoplus_{i\in I}\langle h_i\rangle}.\]
\end{itemize}
Here the bar denotes the closure in $C_B(X)$.
\end{theorem}

\begin{proof}
(1) \textit{implies} (2). Since $\lambda_HX$ ($=\mathfrak{sp}(H)$) is locally compact and paracompact, we have
\[\lambda_HX=\bigcup_{i\in I}U_i,\]
where $\{U_i:i\in I\}$ is a collection of pairwise disjoint open Lindel\"{o}f subspaces of $\lambda_HX$. For every $i$ in $I$ let $H_i$ be the closed subideal of $H$ defined by
\begin{equation}\label{OJU}
H_i=\{h\in H:h_\beta|_{\beta X\setminus U_i}=0\}
\end{equation}
for which we have $\lambda_{H_i}X=U_i$ by Lemma \ref{HJGF}. Note that
\[\lambda_HX=\bigcup_{i\in I}\lambda_{H_i}X=\lambda_{\sum_{i\in I}H_i}X=\lambda_{\overline{\sum_{i\in I}H_i}}X\]
by Lemmas \ref{UEED} and \ref{JH}, which implies that
\begin{equation}\label{HHG}
H=\overline{\sum_{i\in I}H_i}
\end{equation}
by Lemma \ref{HOOF}. First, we verify that
\begin{equation}\label{JB}
\sum_{i\in I}H_i=\bigoplus_{i\in I}H_i.
\end{equation}
We need to verify that
\[H_{i_0}\cap\sum_{i_0\neq i\in I}H_i=0\]
for all $i_0$ in $I$. Fix some $i_0$ in $I$ and let $f_{i_0}$ be in $H_{i_0}\cap\sum_{i_0\neq i\in I}H_i$. Then \[f_{i_0}=f_{i_1}+\cdots+f_{i_n},\]
where $i_j$ is in $I$ and $f_{i_j}$ is in $H_{i_j}$ for all $j=0,1,\ldots,n$. Observe that $f^\beta_{i_j}|_{\beta X\setminus U_{i_j}}=0$ by (\ref{OJU}) for all $j=1,\ldots,n$. This implies that
\[(f^\beta_{i_1}+\cdots+f^\beta_{i_n})|_{\beta X\setminus\bigcup_{j=1}^nU_{i_j}}=0.\]
But $f^\beta_{i_0}=f^\beta_{i_1}+\cdots+f^\beta_{i_n}$ and
\[U_{i_0}\subseteq\beta X\setminus\bigcup_{j=1}^nU_{i_j},\]
thus $f^\beta_{i_0}|_{U_{i_0}}=0$. On the other hand $f^\beta_{i_0}|_{\beta X\setminus U_{i_0}}=0$ by (\ref{OJU}), from which it follows that $f^\beta_{i_0}=0$ and therefore $f_{i_0}=0$. This shows (\ref{JB}). Note that by Lemma \ref{JHFF} for all $i$ in $I$ (since $\lambda_{H_i}X$ ($=U_i$) is Lindel\"{o}f) we have $H_i=\overline{\langle h_i\rangle}$ for some $h_i$ in $H$. Therefore, from (\ref{HHG}) and (\ref{JB}) it follows that
\[H=\overline{\bigoplus_{i\in I}\langle h_i\rangle}.\]

(2) \textit{implies} (1). We have
\[\lambda_HX=\lambda_{\overline{\sum_{i\in I}\langle h_i\rangle}}X=\lambda_{\sum_{i\in I}\langle h_i\rangle}X=\bigcup_{i\in I}\lambda_{\langle h_i\rangle}X=\bigcup_{i\in I}\mathrm{coz}(h_i^\beta)\]
by Lemmas \ref{UEED} and \ref{JH}. Note that $\mathrm{coz}(h_i^\beta)$ and $\mathrm{coz}(h_j^\beta)$ are disjoint, if $i$ and $j$ in $I$ are distinct, as $h_ih_j=0$ (and thus $h^\beta_ih^\beta_j=(h_ih_j)^\beta=0$) by the representation of $H$. Also, note that $\mathrm{coz}(h_i^\beta)$ is $\sigma$-compact (and thus, Lindel\"{o}f) for all $i$ in $I$, as
\[\mathrm{coz}(h_i^\beta)=\bigcup_{n=1}^\infty |h_i^\beta|^{-1}\big([1/n,\infty)\big).\]
Thus $\mathfrak{sp}(H)$ ($=\lambda_HX$) is the union of a collection of pairwise disjoint open Lindel\"{o}f subspaces, and is therefore paracompact.
\end{proof}

\section{On sets of closed ideals of $C_B(X)$}

In this concluding section for a completely regular space $X$ we study certain collections of closed ideals of $C_B(X)$, partially ordered under set-theoretic inclusion.

\begin{definition}\label{KH}
Let $X$ be a completely regular space.
\begin{itemize}
\item Denote
\[\mathscr{U}(X)=\{U:\mbox{$U$ is an open subspace of $\beta X$ containing $X$}\},\]
and
\[\mathscr{H}(X)=\big\{H:\mbox{$H$ is a non-vanishing closed ideal of $C_B(X)$}\big\}.\]
\item Let $H$ be in $\mathscr{H}(X)$. Denote
\[\mathbf{sub}(H)=\{G:\mbox{$G$ is a closed subideal of $H$}\}.\]
\end{itemize}
\end{definition}

Let $(P,\leq)$ and $(Q,\leq)$ be partially ordered sets. A mapping $f:P\rightarrow Q$ is called an \textit{order-homomorphism} if $f(x)\leq f(y)$ whenever $x\leq y$. A mapping $f:P\rightarrow Q$ is called an \textit{order-isomorphism} if it is a bijection and both $f$ and $f^{-1}$ are order-homomorphisms. The partially ordered sets $(P,\leq)$ and $(Q,\leq)$ are \textit{order-isomorphic} if there is an order-isomorphism between them.

\begin{lemma}\label{GFD}
Let $X$ be a completely regular space. Then $(\mathscr{H}(X),\subseteq)$ is order-isomorphic to $(\mathscr{U}(X),\subseteq)$.
\end{lemma}

\begin{proof}
Let the mapping
\[\phi:\big(\mathscr{H}(X),\subseteq\big)\longrightarrow\big(\mathscr{U}(X),\subseteq\big)\]
be defined by
\[\phi(H)=\lambda_HX.\]
Note that by Theorem \ref{TRES} this indeed defines a mapping. Observe that $\phi$ is surjective, as for any $U$ in $\mathscr{U}(X)$ the set
\[H=\big\{f\in C_B(X):f_\beta|_{\beta X\setminus U}=0\big\}\]
is in $\mathscr{H}(X)$ and $\lambda_HX=U$ by a slight modification of the proof which applies to Lemma \ref{HJGF}. (We need to argue that $H$ is non-vanishing. Let $x$ be in $X$. Then $x$ is in $U$. Let $F:\beta X\rightarrow[0,1]$ be a continuous mapping such that $F(x)=1$ and $F|_{\beta X\setminus U}=0$. Let $f=F|_X$. Note that $f_\beta=F$, as the two mappings are continuous and coincide on the dense subspace $X$ of $\beta X$. Then $f$ is in $H$ and $f(x)=F(x)\neq0$.) It is clear that $\phi$ is an order-homomorphism. Also, for any $H_1$ and $H_2$ in $\mathscr{H}(X)$, if $\lambda_{H_1}X\subseteq\lambda_{H_2}X$, then
\[\lambda_{H_2}X=\lambda_{H_1}X\cup\lambda_{H_2}X=\lambda_{H_1+H_2}X=\lambda_{\overline{H_1+H_2}}X,\]
which implies that $H_2=\overline{H_1+H_2}$, and therefore $H_1\subseteq H_2$. (We have used Lemmas \ref{UEED}, \ref{HOOF} and \ref{JH}.) This in particular shows that $\phi$ is an injection (and thus is a bijection) and therefore is an order-isomorphism.
\end{proof}

Let $(P,\leq)$ be a partially ordered set. Then $(P,\leq)$ is a \textit{lattice} if together with every pair $p$ and $q$ of its elements it contains their least upper bound $p\vee q$ and their greatest lower bound $p\wedge q$. Also, $(P,\leq)$ is a \textit{complete upper semi-lattice} if together with every collection $\{p_i:i\in I\}$ of its elements it contains their least upper bound $\bigvee_{i\in I}p_i$.

\begin{theorem}\label{JHGH}
Let $X$ be a completely regular space. Then $(\mathscr{H}(X),\subseteq)$ is a lattice which is also a complete upper semi-lattice. Indeed, for any two elements $H_1$ and $H_2$ of $\mathscr{H}(X)$ and any subcollection $\{H_i:i\in I\}$ of elements of $\mathscr{H}(X)$ we have
\[H_1\wedge H_2=H_1\cap H_2\quad\mbox{and}\quad\bigvee_{i\in I}H_i=\overline{\Big\langle\bigcup_{i\in I}H_i\Big\rangle}.\]
Here the bar denotes the closure in $C_B(X)$.
\end{theorem}

\begin{proof}
This follows from Lemma \ref{GFD}.
\end{proof}

It is known that the order-structure of the set of all closed subspaces of any Hausdorff space (partially ordered under set-theoretic inclusion) determines and is determined by its topology. (See Theorem 11.1 of \cite{B}.) We use this fact in the following.

\begin{theorem}\label{JGSD}
Let $X$ be a completely regular space. Let $H$ and $H'$ be in $\mathscr{H}(X)$. The following are equivalent:
\begin{itemize}
\item[(1)] $\mathfrak{sp}(H)$ and $\mathfrak{sp}(H')$ are homeomorphic.
\item[(2)] $\mathbf{sub}(H)$ and $\mathbf{sub}(H')$ are order-isomorphic.
\end{itemize}
\end{theorem}

\begin{proof}
By Lemma \ref{HJGF} (and Lemmas \ref{UEED}, \ref{HOOF} and \ref{JH}, as we have argued in the proof of Lemma \ref{GFD}) the set of all open subspaces of $\mathfrak{sp}(H)$ and $\mathfrak{sp}(H')$ are order-isomorphic to $\mathbf{sub}(H)$ and $\mathbf{sub}(H')$, respectively, all partially ordered under set-theoretic inclusion. The theorem now follows, as the order-structure of the set of all closed (and thus all open) subspaces of any Hausdorff space determines and is determined by its topology.
\end{proof}

\begin{theorem}\label{IJJ}
Let $X$ and $Y$ be completely regular spaces. Every continuous mapping $f:X\rightarrow Y$ induces a (contravariant) order-homomorphism $\phi:(\mathscr{H}(Y),\subseteq)\rightarrow(\mathscr{H}(X),\subseteq)$ which is injective if $f$ is surjective.
\end{theorem}

\begin{proof}
Let
\[\psi:\big(\mathscr{U}(Y),\subseteq\big)\longrightarrow\big(\mathscr{U}(X),\subseteq\big)\]
be defined by
\[\psi(V)=f_\beta^{-1}(V).\]
Note that this defines a mapping, as $f_\beta^{-1}(V)$ is open in $\beta X$ (since it is the inverse image of an open subspace of $\beta Y$) and contains $X$ (since $V$ contains $Y$). It is clear that $\phi$ preserves order and is therefore an order-homomorphism. Note that if $f$ is surjective then $f_\beta:\beta X\rightarrow\beta Y$ is also surjective. In particular $f_\beta(f_\beta^{-1}(V))=V$ for every $V$ in $\mathscr{U}(Y)$ which shows that $\psi$ is injective. Note that $(\mathscr{H}(Y),\subseteq)$ and $(\mathscr{H}(X),\subseteq)$ are respectively order-isomorphic to $(\mathscr{U}(Y),\subseteq)$ and $(\mathscr{U}(X),\subseteq)$ by Lemma \ref{GFD}. This concludes the proof.
\end{proof}

\subsubsection*{Acknowledgements}

The authors wish to thank the anonymous referee for his/her careful reading of the manuscript, helpful suggestions and comments, and the prompt report sent within less than two weeks from submission of the article.

\end{document}